\documentclass{article}
\usepackage{a4wide}
\usepackage{authblk}
\usepackage[latin1]{inputenc}
\usepackage{amscd}
\usepackage{amsfonts}
\usepackage{amsgen}
\usepackage{amsmath}
\usepackage{amssymb}
\usepackage{amstext}
\usepackage{amsthm}
\usepackage{graphicx}
\usepackage{indentfirst}
\usepackage{latexsym}
\usepackage{bbm}
\usepackage{epsfig}
\usepackage{wrapfig}
\usepackage[all,knot,arc,import,poly]{xy}
\usepackage[usenames]{color}
\usepackage[mathscr]{eucal}
\usepackage{ dsfont } 

\usepackage{MnSymbol}
\newtheorem{Df}{Definition}[section]
\newtheorem{Th}[Df]{Theorem}

\newtheorem{Lem}[Df]{Lemma}
\newtheorem{Ex}[Df]{Example}

\newtheorem{Rmk}[Df]{Remark}

\def\P{{\mathbb P}}
\def\N{{\mathbb N}}
\def\E{{\mathbb E}}

\def\Leb{\textrm{Leb}}

\def\lt{\left}
\def\rt{\right}

\title{Matching strings in encoded sequences}
\author{Adriana Coutinho, Rodrigo Lambert, J\'er\^ome Rousseau}
\date{}
\begin{document}

\maketitle

\begin{abstract}
We investigate the length of the longest common substring for encoded sequences and its asymptotic behaviour. The main result is a strong law of large numbers for a re-scaled version of this quantity, which presents an explicit relation with the R\'enyi entropy of the source. We apply this result to the zero-inflated contamination model and the stochastic scrabble. In the case of dynamical systems, this problem is equivalent to the shortest distance between two observed orbits and its limiting relationship with the correlation dimension of the pushforward measure. An extension to the shortest distance between orbits for random dynamical systems is also provided. 
\end{abstract}
\begin{center}
\textbf{keywords: }{string matching, R\'enyi entropy, shortest distance, correlation dimension, random dynamical systems, coding}
\end{center}
\begin{center}
\end{center}

\section{Introduction}
Finding patterns on symbolic strings has been a widely studied subject matter on Genetics, Probability and Information Theory over the years. The investigations about how much information a $n$-string has on the whole realization of the process are naturally linked with the concept of redundancy and compression algorithms. On the other hand, the overlap between (some proportion of) two different strings can give us some knowledge about the similarity of the sources that generate those processes. Moreover, repetition and similarity are two well-exploited concepts in the study of DNA sequences.

In view of repetition, one of the earliest studied quantities was the well-known Ornstein-Weiss return time. A strong law of large numbers for this quantity and its explicit relationship with the entropy of the source was stated in \cite{OW} 
and the convergence in distribution has been widely studied (see e.g. the reviews \cite{AbadiGalves,Saussol,Hay}). 
An interesting and intuitive link between return times and the notion of data compression schemes can be found in \cite{WZ}, and a consistent estimator for the entropy based on that quantity was provided in \cite{KASW}. We remark also the first return of a string to its own $n$-cylinder (which is an outspread of the return times investigation), which can be found in \cite{ACS,STV,AbVa, haydnvaienti,AbLa1,GrKu,AbCa,AbGaRa}, and references therein.

On the other hand, the notion of coincidence has been exploited on the context of \emph{waiting times} \cite{DeKo,Wyner,WZ}. In \cite{Wyner} it was proved an exponential limiting distribution for the waiting time (properly re-scaled), when the source-measure is $\psi$-mixing with exponential decay of correlations. We recall that in this paper the author considered two independent copies of the same process. %

On the Erd\"os-R\'enyi scenario, the similarity between two sources has also been widely investigated (we refer the reader to \cite{DeKaZe1, DeKaZe2, Manson, Neuhauser} and references therein). As a recent example, we recall the shortest path between two observables defined in \cite{AbLa2}. 
In this work the authors  proved an almost-sure linear increasing, a large deviation principle and a weak convergence for the shortest path function. All these results were linked to the divergence between two measures, which is essentially a measure of similarity between the two source-measures. 

Holding on the same scenario, a remarkable matching quantity has been studied in \cite{ArWa}: $M_n (x,y)$, the length of the \emph{longest matching consecutive subsequence} (or longest common substring) between two sequences. 
More precisely, if $x$ and $y$ are two realizations of the stochastic processes $(X_n)_{n \in \mathbb{N}}$ and $(Y_n)_{n \in \mathbb{N}}$,
\begin{equation*}
M_n (x,y)= \displaystyle \max \left\{k \ : x_i^{i+k-1} = y_j^{j+k-1} \ \mbox{for some} \ 0 \leq i,j \leq n-k\right\} \ ,
\end{equation*}
where $x_i^{i+k-1}$ (respectively $y_j^{j+k-1}$) denotes the substring $x_ix_{i+1}\cdots x_{i+k-1}$ (respectively $y_jy_{j+1}\cdots y_{j+k-1}$).

If the two processes are independent and identically distributed, and generated by the same source $\mathds{P}$, the authors proved that $M_n/(\log_{1/p}n) \to 2$ for almost every realization $(x,y)$, with $p=\mathds{P}(X_0=Y_0)$ \cite{ArWa}.
Furthermore, they also proved that the same result holds for Markov chains,  but with $p$ being the largest eigenvalue of the matrix $[(p_{ij})^2]$, where $[p_{ij}]$ is the transition matrix. This result was recently generalized in \cite{BaLiRo} for $\alpha$-mixing processes with exponential decay and $\psi$-mixing processes with polynomial decay with a limit depending on the R\'enyi entropy of $\P$ {and in \cite{LCS-random} for random sequences in random environment.}
We recall that weak convergence theorems for sequence matching where also investigated over the last years (e.g \cite{Manson, Neuhauser}).

Further generalizations of such quantity has also appeared on the literature.
An interesting example was the sequence matching with scores introduced in \cite{ArMoWa}. In this paper, the authors consider that each symbol in the alphabet has a particular score (or weight). Therefore, each match score becomes a function which depends on the match size and the weights of the symbols as well. In the iid case and for Markov chains, they obtain a strong law of large numbers for the highest-scoring matching substring. A generalization for this statement (allowing incomplete matches, for instance) can be found in \cite{DeKaZe1, DeKaZe2}.

Following the direction of the pattern investigation between strings, one can ask if some of the above mentioned results hold if we transform our sequences following certain rules of modification. In other words: what happens if we consider encoded
sequences as our interest objects of investigation?

In this paper we study a version of the longest matching substring problem when the orbits are encoded by a measurable function (which we call \emph{{encoder}} or \emph{observation}, depending on the context). We call it the \emph{longest common substring between encoded strings}. More precisely, let $\chi$ (respectively $\tilde\chi$) be an alphabet, $\chi^{\mathds{N}}$ (respectively $\tilde\chi^{\mathds{N}}$) the space of all sequences with symbols in $\chi$ (respectively $\tilde\chi$) and let $f: \chi^{\mathds{N}} \to \tilde\chi^{\mathds{N}}$ be {a measurable function (following the terminology of \cite{KeSu}, we will call $f$ an encoder (one can also see \cite{Shields} where $f$ is called a coder))}.  
Given two sequences $x,y \in \chi^{\mathds{N}},$ we define the $n$-length of the longest common substring for the encoded pair $(f(x),f(y))$ by
\begin{equation*}
M_n^f (x,y)= \displaystyle \max \left\{k \ : f(x)_i^{i+k-1} = f(y)_j^{j+k-1} \ \mbox{for some} \ 0 \leq i,j \leq n-k\right\} \ ,
\end{equation*}
where $f\left(x\right)_i^{i+k-1}$ and $f\left(y\right)_j^{j+k-1}$ denotes the substrings (of the encoded sequences $f(x)$ and $f(y)$) of length $k$ beginning in $f(x)_i$ and $f(x)_j$ respectively.

In the symbolic case, we prove an almost sure convergence for $M_n^f$. 
Namely, we provide necessary conditions on the {encoder} as well as in the source to prove that $M_n^f$ grows logarithmically fast in $n$. It is in fact a law of large numbers with limiting rate linked with the R\'enyi entropy of the pushforward measure (denoted by ${H}_2(f_*\mathds{P})$). Namely, if $\mathds{P}$ is the source-measure, then
\begin{equation}\label{mr1}\tag{\footnotesize{$\largestar$}}
\lim_{n \to \infty}\frac{M_n^f(x,y)}{\log n} = \frac{2}{{H}_2(f_*\mathds{P})} \ \ \ \ \ \ \ \ \ \  \mathds{P} \otimes \mathds{P} \mbox{-a.s.}
\end{equation}

In the context of stochastic coding, \eqref{mr1} shows to be rather applicable. As a first illustration of this feature, we generalize the results from the \emph{stochastic scrabble} given by \cite{ArMoWa}, from a Markov chain to a general $\alpha$-mixing process with exponential decay. The second application deals with the stochastic noise (or contamination {encoder}), which can be viewed in \cite{CoGaLe,GaMo}.

Recently, \cite{BaLiRo} showed that the problem of the longest common substring for stochastic processes is related to the shortest distance between two orbits and, to the best of our knowledge, this was the first article where this quantity was defined and studied.
Following this idea, we can observe that, in dynamical systems, the correspondent of the longest common substring for the encoded
pair is the shortest distance between observed orbits. Formally, let $f:X \rightarrow Y$ be a measurable function, called the observation. If we consider a dynamical system $(X, T, \mu)$, we investigate the asymptotic behavior of
\begin{equation}\label{shortdist}
m_n^f(x,y) = \min_{i,j=0,\ldots,n-1}\left(d(f(T^ix),f(T^jy))\right) \ ,
\end{equation}
and prove that its limiting behavior is related to the correlation dimension of the pushforward measure $f_*\mu$ (denoted ${C}_{f_*\mu}$). If $f$ is the identity in $X$, we recover the shortest distance between two orbits problem, studied in \cite{BaLiRo}. In that paper the authors provide a law of large numbers and related it with the correlation dimension of the source measure. In the present paper we generalize this result for a family of {observations}, concluding that the limiting rate is given by the dimension of the pushforward measure $f_*\mu$ (under suitable conditions on $f$). Namely, for rapidly mixing systems,
\begin{equation}\label{mr2}\tag{\footnotesize{$\largestar \largestar$}}
\underset{n \rightarrow \infty}{\lim}\frac{\log m_n^f(x,y)}{-\log n} = \frac{2}{{C}_{f_*\mu}} \ \ \ \ \ \ \ \ \ \  \mu \otimes \mu  \mbox{-a.s.} \ ,
\end{equation}
provided that ${C}_{f_*\mu}$ exists.

In \cite{RS, MaRousseau,Rousseau}, the study of observed orbits (in particular, the study of return and hitting time) was used to obtain results for random dynamical systems. Following this idea, we combine $\eqref{mr2}$ with a particular observation $f$ to obtain the following strong law of large numbers for random dynamical systems (provided that $C_{\nu}$ exists)
\begin{equation}\label{mr3}\tag{\footnotesize{$\largestar \largestar \largestar$}}
\underset{n \rightarrow \infty}{\lim}\frac{\log m_n^{\omega, \tilde{\omega}}(x,\tilde{x})}{-\log n} = \frac{2}{C_{\nu}} \ \ \ \ \ \ \ \ \ \  \mu \otimes \mu  \mbox{-a.s.} \ ,
\end{equation}
where $m_n^{\omega, \tilde{\omega}}(x,\tilde{x})$ is the shortest distance between two random orbits ($\{x, T_\omega x,...,T^n_\omega x\}$ and $\{\tilde{x}, T_{\tilde{\omega}} \tilde{x},...,T^n_{\tilde{\omega}} \tilde{x}\}$) and $C_{\nu} $ is the correlation dimension of the stationary measure (we refer the reader to section \ref{random} for more details).
We present then a collection of applications for this statement. The first one treats non i.i.d. random dynamical systems. The second deals with random perturbed dynamics. We finish the applications with random hyperbolic toral automorphisms.

The rest of this paper is organized as follows. In Section \ref{renyi} we study the relation between the longest common substring for encoded sequences and the R\'enyi entropy. We state precisely result \eqref{mr1} and apply it to the stochastic scrabble and the the zero-inflated contamination model. In Section \ref{shortest}, we analyse the behaviour of the shortest distance between observed orbits of a dynamical system and present result \eqref{mr2}. Section \ref{random} deals with the case of random dynamical systems and states result \eqref{mr3}, as well as some applications.

\section{Reaching R\'enyi entropy via string matching of encoded sequences}\label{renyi}

The present section is dedicated to study of the longest common substring of encoded sequences. We start by presenting some terminology and definitions, in order to introduce the problem.

Let $(\Omega, \mathcal{F}, \mathds{P})$ be a probability space, where $\Omega=\chi^{\mathds{N}}$ for some alphabet $\chi$, $\mathcal{F}$ the sigma-algebra generated by the $n$-cylinders in $\Omega$ and $\mathds{P}$ is a stationary probability measure on $\mathcal{F}$. If $\sigma$ is the left shift on $\Omega$, we can see $(\Omega, \mathcal{F}, \mathds{P}, \sigma)$ as a symbolic dynamical system with $\mathds{P}$ $\sigma$-invariant. Let  $\tilde\Omega=\tilde\chi^{\mathds{N}}$ for some alphabet $\tilde\chi$ and $\tilde{\mathcal{F}}$ the sigma-algebra generated by the $n$-cylinders in $\tilde\Omega$.

\begin{Df}\label{longestdist}
Let $f: \Omega \to \tilde\Omega $ be {an encoder}. Given two sequences $x,y \in \Omega,$ we define the $n$-length of the longest common substring for the encoded pair $(f(x),f(y))$ by
$$
M_n^f (x,y)= \displaystyle \max \left\{k \ : f(x)_i^{i+k-1} = f(y)_j^{j+k-1} \ \mbox{for some} \ 0 \leq i,j \leq n-k\right\} \ ,
$$
where $f(x)_i^{i+k-1}$ and $f(y)_j^{j+k-1}$ denote the substrings of length $k$ beginning in $f(x)_i$ and $f(y)_j$ respectively.
\end{Df}

For $y \in \Omega$ (respectively $\tilde\Omega $) we denote by $C_n(y)$ the $n$-cylinder containing $y$, that is, the set of sequences $z \in \Omega$ (respectively $\tilde\Omega $) such that $z_i=y_i$ for any $i=0, \ldots, n-1$. We denote $\mathcal{F}_0^n$ (respectively $\tilde{\mathcal{F}}_0^n$) the sigma-algebra on $\Omega$ (respectively $\tilde\Omega $) generated by all $n$-cylinders.
\begin{Df}
The lower and upper R\'enyi entropies of a measure $\mathds{P}$ are defined as
$$\underline{H}_2(\mathds{P}) = -\displaystyle\underset{k \to \infty}{\underline{\lim}}\frac{1}{k}\log\sum\limits_{C_k} \mathds{P}(C_k)^2 \ \  \mbox{and} \ \ \overline{H}_2(\mathds{P}) = -\displaystyle\underset{k \to \infty}{\overline{\lim}}\frac{1}{k}\log\sum\limits_{C_k} \mathds{P}(C_k)^2 \ ,$$
where the sums are taken over all k-cylinders. When the limit exists we denote by ${H}_2(\mathds{P})$ the common value.
\end{Df}
In general, the existence of the R\'enyi entropy is not known. However, it was computed in some particular cases: Bernoulli shift, Markov chains and Gibbs measure of a H\"older-continuous potential \cite{haydnvaienti}. The existence was also proved for $\phi$-mixing measures \cite{luc}, for weakly $\psi$-mixing processes \cite{haydnvaienti} and for $\psi_g$-regular processes \cite{AbCa}. In section \ref{secrenyi}, we will prove that for Markov chains, the R\'enyi entropy does not depend on the initial distribution but only on the transition matrix and that one can compute the R\'enyi entropy even if the measure is not stationary.

\begin{Df}
Consider the dynamical system $(\Omega, \mathds{P}, \sigma)$. We say that it is $\alpha$-mixing if there exists a function $\alpha: \mathds{N} \to \mathds{R}$ where $\alpha(g)$ converges to zero when $g$ goes to infinity and such that
\begin{equation}\label{alpha}
\sup_{A \in \mathcal{F}_0^n \ ; \  B \in \mathcal{F}_0^m}\lt|\mathds{P}\lt( A \cap \sigma^{-g-n}B\rt) -\mathds{P}(A) \mathds{P}(B)\rt| \leq \alpha(g) \ ,
\end{equation}
for all $m,n \in \mathds{N}$.

We say that the system is $\psi$-mixing if there exists a function $\psi: \mathds{N} \to \mathds{R}$ where $\psi(g)$ converges to zero when $g$ goes to infinity and such that
\begin{equation}\label{psi}
\sup_{A \in \mathcal{F}_0^n \ ; \  B \in \mathcal{F}_0^m}\lt|\frac{\mathds{P}\lt( A \cap \sigma^{-g-n}B\rt) -\mathds{P}(A) \mathds{P}(B)}{\mathds{P}(A) \mathds{P}(B)}\rt| \leq \psi(g),
\end{equation}
for all $m,n \in \mathds{N}$. In the cases that $\alpha(g)$ or $\psi(g)$ decreases exponentially fast to zero, we say that the system has an exponential decay.
\end{Df}

Now we are ready to present the main result of this section. It states that, under suitable conditions and large values of $n$, the longest common substring behaves like $\log n$, for almost all realizations.

\begin{Th} \label{discrete} Consider $f:\Omega \to \tilde\Omega$ {an encoder} such that $\underline{H}_2(f_*\mathds{P})>0$. For $\mathds{P} \otimes \mathds{P}$-almost every $(x,y) \in \Omega \times \Omega$,
\begin{eqnarray}\label{ineq1discretecase}
\displaystyle\underset{n \to \infty}{\overline{\lim}}\frac{M_n^f(x,y)}{\log n} \leq \frac{2}{\underline{H}_2(f_*\mathds{P})} \cdot
\end{eqnarray}
Moreover, if
\begin{description}
\item [(i)] the system $(\Omega, \mathds{P}, \sigma)$ is $\alpha$-mixing with an exponential decay (or $\psi$-mixing with $\psi(g)=g^{-a}$ for some $a>0$);
\item [(ii)] $C_n \in \tilde{\mathcal{F}}_0^n$ implies $f^{-1}C_n \in \mathcal{F}_0^{h(n)},$ where $h(n)= o(n^{\gamma})$, for some $\gamma>0$,
\end{description}
then, for $\mathds{P} \otimes \mathds{P}$-almost every $(x,y) \in \Omega \times \Omega$,
\begin{eqnarray}\label{ineq2discretecase}
\displaystyle\underset{n \to \infty}{\underline{\lim}}\frac{M_n^f(x,y)}{\log n} \geq \frac{2}{\overline{H}_2(f_*\mathds{P})} \cdot
\end{eqnarray}
Therefore, if the R\'enyi entropy exists, we get for $\mathds{P} \otimes \mathds{P}$-almost every $(x,y) \in \Omega \times \Omega$,
\begin{equation}\tag{\footnotesize{$\largestar$}}
\lim_{n \to \infty}\frac{M_n^f(x,y)}{\log n} = \frac{2}{{H}_2(f_*\mathds{P})} \cdot 
\end{equation}
\end{Th}
{\begin{Rmk}We emphasize that to obtain this result one cannot apply directly Theorem 7 of \cite{BaLiRo} since in general the pushforward measure $f_*\mathds{P}$ is not stationary (see e.g. Section \ref{secscrabble}).
\end{Rmk}
}
\begin{proof}
For simplicity we assume $\alpha(g)=e^{-g}.$ The $\psi$-mixing case can be obtained by a simple modification.
The proof of this theorem follows the lines of the proof of the Theorem 7 in \cite{BaLiRo}, {but an extra care is needed (mainly in the second part of the proof) since we are working with pre-image of cylinders (instead of cylinders in \cite{BaLiRo}). }

In the first part of the proof, for $\epsilon >0$ we denote
$$k_n=\left\lceil\frac{2 \log n + \log \log n}{\underline{H_2}(f_*\mathds{P})-\epsilon}\right\rceil.$$
Let us also denote
$$A_{i,j}^f(y)=\sigma^{-i}[f^{-1}C_{k_n}(f(\sigma^jy))]$$
and
$$S_n^f(x,y)= \sum_{i,j= 1, \ldots, n}\mathds{1}_{A_{i,j}^f(y)}(x).$$

We first show that the event $\left\{M_n^f \geq k_n \right\}$ occurs only finitely many times. It follows from definition of $S_n^f$ and Markov's inequality that
$$\mathds{P} \otimes \mathds{P} \lt(\lt\{(x,y): M_n^f(x,y) \geq k_n\rt\}\rt)=\mathds{P} \otimes \mathds{P} \lt(\lt\{(x,y): S_n^f(x,y)\geq1\rt\}\rt)  \leq  \E\lt(S_n^f\rt).$$
Moreover, by computing the expected value of $S_n^f$ we get
\begin{eqnarray*}
\E\lt(S_n^f\rt) & = & \int \int \displaystyle\sum_{i,j=1,\ldots,n} \mathbbm{1}_{A_{ij}^f(y)}(x) \ d\mathds{P}(x) \ d\mathds{P}(y) \\
& = & \displaystyle\sum_{i,j=1,\ldots,n} \int \mathds{P}\left(f^{-1}C_{k_n}(f(\sigma^jy))\right) \ d\mathds{P}(y) \\
& =& n^2 \int f_*\mathds{P}\left(C_{k_n}(f(y))\right) \ d\mathds{P}(y).
\end{eqnarray*}

Thus,
$$\mathds{P} \otimes \mathds{P} \lt(\lt\{(x,y): M_n^f(x,y) \geq k_n\rt\}\rt)  \leq   n^2 \int f_*\mathds{P}\left(C_{k_n}(f(y))\right) \ d\mathds{P}(y).$$
For large values of $n$, by definition of $\underline{H}_2(f_*\mathds{P})$ it holds
$$\int f_*\mathds{P}\left(C_{k_n}(f(y))\right)\ d\mathds{P}(y)= \sum_{C_{k_n}}f_*\mathds{P}\left(C_{k_n}\rt)^2 \leq e^{-k_n(\underline{H}_2(f_*\mathds{P})-\epsilon)}.$$
Moreover by definition of $k_n,$
$$
\mathds{P} \otimes \mathds{P} \lt(\lt\{(x,y): M_n^f(x,y) \geq k_n\rt\}\rt) \leq n^2 e^{-k_n (\underline{H}_2(f_*\mathds{P})-\epsilon)} \leq\frac{1}{\log n} \ .
$$

Choosing a subsequence $\{n_{\kappa}\}_{\kappa \in \mathds{N}}$ such that $n_{\kappa}= \lceil e^{\kappa^2}\rceil$ we have that
$$
\mathds{P} \otimes \mathds{P} \lt(\lt\{(x,y): M_{n_{\kappa}}^f(x,y) \geq k_{n_{\kappa}}\rt\}\rt) \leq \frac{1}{\kappa^2} \ .
$$

Since the last quantity is summable in $\kappa$, the Borel-Cantelli lemma gives that if $\kappa$ is large enough, then for almost every pair $(x,y)$ it holds

$$M_{n_{\kappa}}^f(x,y) < k_{n_{\kappa}}$$
and  then
\begin{eqnarray}\label{desigbc}
\frac{ M_{n_{\kappa}}^f(x,y)}{\log n_{\kappa}} \leq  \frac{1}{\underline{H}_2(f_*\mathds{P})-\epsilon}\lt( 2 + \frac{1+\log \log n_{\kappa}}{\log n_{\kappa}}\rt).
\end{eqnarray}

We observe that for all $n$, there exists $\kappa$ such that $e^{\kappa} \leq n \leq e^{\kappa+1}.$ In addition, since $\lt(M_n^f\rt)_{n \in \mathds{N}}$ is an increasing sequence, we get
\begin{equation}\label{log}
\frac{ M_{n_{\kappa}}^f(x,y)}{ \log n_{\kappa+1}} \leq \frac {M_n^f(x,y)}{ \log n}  \leq \frac {M_{n_{\kappa+1}}^f(x,y)}{ \log n_{\kappa}} \ .
\end{equation}
Taking the limit superior in the above inequalities and observing that $\lim\limits_{\kappa \to \infty} \dfrac{\log n_{\kappa}}{ \log n_{\kappa+1}}=1$ by \eqref{log} we obtain
$$
\underset{n \rightarrow \infty}{\overline{\lim}} \frac{M_n^f(x,y)}{\log n} = \underset{\kappa \rightarrow \infty}{\overline{\lim}} \frac{M_{n_{\kappa}}^f(x,y)}{\log n_{\kappa}}.
$$

Thus, by \eqref{desigbc} we have
$$
\underset{n \rightarrow \infty}{\overline{\lim}} \frac{M_n^f(x,y)}{\log n}  \leq \frac{2}{\underline{H}_{2}(f_*\mathds{P})-\epsilon}.
$$
Since $\epsilon$ can be arbitrarily small, \eqref{ineq1discretecase} is proved.

{Despite some similarities with Theorem 7 in \cite{BaLiRo}, we emphasize that second part of the present proof is quite different, in particular since the length of the encoded sequences may be changed by the encoder.} 

We will now prove \eqref{ineq2discretecase}. In order to do that denote, for $\epsilon>0$,
$$k_n=\left\lfloor\frac{2 \log n + b \log \log n}{\overline{H_2}(f_*\mathds{P})+\epsilon}\right\rfloor$$
where $b$ is a constant to be chosen.

Note that by definition of $S_n^f$ we have
\begin{eqnarray*}
\mathds{P} \otimes \mathds{P} \lt(\lt\{(x,y): M_n^f(x,y) < k_n\rt\}\rt) &= &  \mathds{P} \otimes \mathds{P} \lt(\lt\{(x,y): S_n^f(x,y)=0\rt\}\rt) \\
 &\leq &  \mathds{P} \otimes \mathds{P} \lt(\lt\{(x,y): \lt|S_n^f(x,y)- \E\lt(S_n^f\rt)\rt| \geq \lt|\E\lt(S_n^f\rt)\rt|\rt\}\rt).
\end{eqnarray*}
By Chebyshev's inequality we deduce that 
$$\mathds{P} \otimes \mathds{P} \lt(\lt\{(x,y): M_n^f(x,y) < k_n\rt\}\rt) \leq  \frac{\textrm{var} \lt(S_n^f\rt)}{\E\lt(S_n^f\rt)^2}.$$
We have to estimate the variance of $S_n^f.$

We see at once that
\begin{eqnarray}\label{vardis}
\textrm{var}\lt(S_n^f\rt) &  =  & \sum_{1 \leq i,i',j,j' \leq n} \textrm{\textrm{cov}}\left(\mathbbm{1}_{A_{ij}^f},\mathbbm{1}_{A_{i'j'}^f}\right) \nonumber\\
&=& \sum_{1 \leq i, i', j, j' \leq n} \int \int \mathbbm{1}_{f^{-1}C_{k_n}(f(\sigma^jy))}(\sigma^ix) \mathbbm{1}_{f^{-1}C_{k_n}(f(\sigma^{j'}y))}(\sigma^{i'}x)  \\
& -& n^4 \lt(\sum_{C_{k_n}} f_*\mathds{P}\lt(C_{k_n} \rt)^2 \rt)^2.\nonumber
\end{eqnarray}

Let $g=g(n)= (\log n)^{\beta}$, for some $\beta>\max\{1,\gamma\}$. There are four cases to consider.

Case 1: $\lt|i-i'\rt| > g+ k_n$.\label{case1} Using the $\alpha$-mixing condition we have
\begin{eqnarray}\label{eq1}
& & \int \lt(\int \mathbbm{1}_{f^{-1}\lt(C_{k_n}(f(\sigma^j y))\rt)}(\sigma^{(i-i')}x) \mathbbm{1}_{f^{-1}\lt(C_{k_n}(f(\sigma^{j'} y))\rt)}( x) d\mathds{P}(x) \rt) d\mathds{P}(y) \nonumber \\
& \leq &\alpha(g+k_n-h(k_n)) +\int \lt(\int \mathbbm{1}_{f^{-1}\lt(C_{k_n}(f(\sigma^j y))\rt)}(x) d\mathds{P}(x) \int \mathbbm{1}_{f^{-1}\lt(C_{k_n}(f(\sigma^{j'} y))\rt)}( x) d\mathds{P}(x) \rt) d\mathds{P}(y) \nonumber \\
& =& \alpha(g+k_n-h(k_n))+\int f_*\mathds{P}\lt(C_{k_n}\lt(f\lt(\sigma^j y\rt)\rt)\rt) f_*\mathds{P}\lt(C_{k_n}\lt(f\lt(\sigma^{j'} y\rt)\rt)\rt) d\mathds{P}(y).  \nonumber \\
\end{eqnarray}

To estimate the first term of the sum above we analyse two cases.

Case 1.1: $\lt|j-j'\rt| > g+ k_n$.\label{case11} In this case we have
\begin{eqnarray*}
&&\int f_*\mathds{P}\lt(C_{k_n}\lt(f\lt(\sigma^j y\rt)\rt)\rt) f_*\mathds{P}\lt(C_{k_n}\lt(f\lt(\sigma^{j'} y\rt)\rt)\rt) d\mathds{P}(y) \nonumber \\
&=& \int f_*\mathds{P}\lt(C_{k_n}\lt(f\lt(\sigma^{j-j'} y\rt)\rt)\rt) f_*\mathds{P}\lt(C_{k_n}\lt(f\lt( y\rt)\rt)\rt) d\mathds{P}(y \nonumber)\\
&=&\sum_{C_{k_n}, C_{k_n}^{'}}\int_{f^{-1}\lt(C_{k_n}\rt) \cap \sigma^{j-j'}\lt(f^{-1}\lt(C_{k_n}^{'}\rt)\rt)}f_*\mathds{P}\lt(C_{k_n}\rt)f_* \mathds{P}\lt(C_{k_n}^{'}\rt) d\mathds{P}(y) \nonumber \\
&=& \sum_{C_{k_n}, C_{k_n}^{'}} f_*\mathds{P}\lt(C_{k_n}\rt) f_*\mathds{P}\lt(C_{k_n}^{'}\rt) \mathds{P}\lt(f^{-1}\lt(C_{k_n}\rt) \cap \sigma^{j-j'}\lt(f^{-1}\lt(C_{k_n}^{'}\rt)\rt)\rt).
\end{eqnarray*}

Using the $\alpha$-mixing condition in the last expression we get that
\begin{eqnarray}\label{eq11}
&&\int f_*\mathds{P}\lt(C_{k_n}\lt(f\lt(\sigma^j y\rt)\rt)\rt) f_*\mathds{P}\lt(C_{k_n}\lt(f\lt(\sigma^{j'} y\rt)\rt)\rt) d\mathds{P}(y) \nonumber \\
&\leq &  \sum_{C_{k_n}, C_{k_n}^{'}}f_* \mathds{P}\lt(C_{k_n}\rt) f_*\mathds{P}\lt(C_{k_n}^{'}\rt) \lt( f_*\mathds{P}\lt(C_{k_n}\rt) f_*\mathds{P}\lt(C_{k_n}^{'}\rt)\rt) \nonumber \\
&& + \sum_{C_{k_n}, C_{k_n}^{'}}f_* \mathds{P}\lt(C_{k_n}\rt) f_*\mathds{P}\lt(C_{k_n}^{'}\rt) \lt( \alpha\lt(g +k_n- h(k_n)\rt)\rt) \nonumber \\
& =& \alpha\lt(g +k_n- h(k_n)\rt) + \lt(\sum_{C_{k_n}} f_*\mathds{P}\lt(C_{k_n} \rt)^2 \rt)^2.
\end{eqnarray}

Case 1.2\label{case12} $\lt|j-j'\rt| \leq g+ k_n$. By H\"older's inequality it follows that,

\begin{eqnarray}\label{eq12}
&&\int f_*\mathds{P}\lt(C_{k_n}\lt(f\lt(\sigma^j y\rt)\rt)\rt) f_*\mathds{P}\lt(C_{k_n}\lt(f\lt(\sigma^{j'} y\rt)\rt)\rt) d\mathds{P}(y)  \nonumber\\
& \leq & \lt( \int f_*\mathds{P}\lt(C_{k_n}\lt(f\lt(\sigma^j y\rt)\rt)\rt)^2 d\mathds{P}(y)\rt)^{1/2}\lt( \int f_*\mathds{P}\lt(C_{k_n}\lt(f\lt(\sigma^{j'} y\rt)\rt)\rt)^2 d\mathds{P}(y)\rt)^{1/2}  \nonumber\\
&=& \sum_{C_{k_n}} f_*\mathds{P}\lt(C_{k_n}\rt)^3\nonumber\\
& \leq & \lt( \sum_{C_{k_n}}f_*\mathds{P}\lt(C_{k_n}\rt)^2 \rt)^{3/2}.
\end{eqnarray}
where the last inequality comes from the subadditivity of the function $z(x)=x^{2/3}$.

Case 2. $\lt|i-i'\rt| \leq g+ k_n$: \label{case2} \\

Case 2.1 $\lt|j-j'\rt| > g+ k_n$: \\

By symmetry, this case is analogous to the case 1.2. \\ 

Case 2.2. $\lt|j-j'\rt| \leq g+ k_n$:\label{case22}
\begin{eqnarray}\label{eq22}
&&\int \int \mathbbm{1}_{f^{-1}\lt(C_{k_n}(f(\sigma^j y))\rt)}(\sigma^ix) \mathbbm{1}_{f^{-1}\lt(C_{k_n}(f(\sigma^{j'} y))\rt)}(\sigma^{i'}x) d\mathds{P}(x) d\mathds{P}(y)   \nonumber  \\
&\leq & \int \int \mathbbm{1}_{f^{-1}\lt(C_{k_n}(f(\sigma^j y))\rt)}(\sigma^ix) d\mathds{P}(x) d\mathds{P}(y)  \nonumber  \\
&=& \sum_{C_{k_n}} f_*\mathds{P}\lt(C_{k_n}\rt)^2 .
\end{eqnarray}

Putting the estimates \eqref{eq11},\eqref{eq12}, \eqref{eq22}  together in \eqref{vardis} we get
\begin{eqnarray}\label{pxp}
\frac{\textrm{var}(S_n^f)}{\E(S_n^f)^2} &\leq & \frac{ 3n^4\alpha\lt(g +k_n- h(k_n)\rt) + 4n^3(g+k_n)\lt(\sum\limits_{C_{k_n}}f_*\mathds{P}\lt(C_{k_n}\rt)^2 \rt)^{3/2}}{\lt(n^2 \sum\limits_{C_{k_n}}  f_*\mathds{P}\lt(C_{k_n}\rt)^2\rt)^2} \nonumber \\
&+& \frac{ 4n^2(g+k_n)^2 \sum\limits_{C_{k_n}}  f_*\mathds{P}\lt(C_{k_n}\rt)^2}{\lt(n^2 \sum\limits_{C_{k_n}}  f_*\mathds{P}\lt(C_{k_n}\rt)^2\rt)^2}.
\end{eqnarray}

We estimate each term on the right separately. Using the definition of $k_n$ and of the R\'enyi entropy, for $n$ large enough, we have for the first term

\begin{eqnarray*}
\frac{ 3n^4\alpha\lt(g +k_n- h(k_n)\rt)}{\lt(n^2 \sum\limits_{C_{k_n}}  f_*\mathds{P}\lt(C_{k_n}\rt)^2\rt)^2}
& \leq & \frac{3n^4\alpha\lt(g +k_n- h(k_n)\rt)}{(\log n)^{-2b}}.
\end{eqnarray*}

By hypothesis, $h(k_n)=o((\log n)^{\gamma})$. Therefore, by definition of $g$ and $k_n$, for n large enough we have $g +k_n- h(k_n)>\log(n^4)$. Hence,
\begin{eqnarray}\label{term1s}
\frac{ 3n^4\alpha\lt(g +k_n- h(k_n)\rt)}{\lt(n^2 \sum\limits_{C_{k_n}}  f_*\mathds{P}\lt(C_{k_n}\rt)^2\rt)^2} & \leq& 3(\log n)^{2b}.
\end{eqnarray}

To estimate the second term we obtain

\begin{eqnarray}\label{term2s}
&&\frac{ 4n^3(g+k_n)\lt(\sum\limits_{C_{k_n}}f_*\mathds{P}\lt(C_{k_n}\rt)^2 \rt)^{3/2}}{\lt(n^2 \sum\limits_{C_{k_n}}  f_*\mathds{P}\lt(C_{k_n}\rt)^2\rt)^2} \nonumber \\
& \leq &  4(g+k_n)(\log n)^{b/2} \nonumber \\
& \leq & 4(\log n)^{\beta+ b/2} + \frac{2 (\log n)^{1+b/2} + b \log (\log n)(\log n)^{b/2}}{\overline{H}_2(f_*\mathds{P})+ \epsilon}.
\end{eqnarray}

Finally for the third term we get
\begin{eqnarray}\label{term3s}
&&\frac{ 4n^2(g+k_n)^2 \sum\limits_{C_{k_n}} f_*\mathds{P}\lt(C_{k_n}\rt)^2}{\lt(n^2 \sum\limits_{C_{k_n}}  f_*\mathds{P}\lt(C_{k_n}\rt)^2\rt)^2} \nonumber \\
&\leq & 4(g+k_n)^2 (\log n)^b \nonumber \\
& \leq & 8(\log n)^{2\beta+ b} + \frac{4 (\log n)^{2+b} + b^2 (\log (\log n))^2(\log n)^{b}}{(\overline{H}_2(f_*\mathds{P})+ \epsilon)^2}.
\end{eqnarray}

Taking $b <-4\beta$ and substituting \eqref{term1s}, \eqref{term2s} and \eqref{term3s} into \eqref{pxp}, we obtain
\begin{eqnarray}
\mathds{P} \otimes \mathds{P} \lt(\lt\{ (x,y): M_n^f(x,y) < k_n\rt\}\rt) \leq \mathcal{O}\lt((\log n)^{-1}\rt).
\end{eqnarray}
Thus, taking a subsequence $\{n_{\kappa}\}_{\kappa}= \lceil e^{\kappa^2}\rceil$ as in the proof of \eqref{ineq1discretecase}, we can use Borel-Cantelli Lemma to obtain \eqref{ineq2discretecase}.

Finally, if the R\'enyi entropy exists, by \eqref{ineq1discretecase} and \eqref{ineq2discretecase} we conclude the proof of the theorem.
\end{proof}




In what follows we compute the R\'enyi entropy for Markov chains and then we apply the above stated theorem to some well-known cases of probability's literature. The first one is a contamination {encoder} that flips to zero some symbols of the sequence and the second one gives a weight on each symbol of $\chi$. 

\subsection{R\'enyi entropy of Markov chains}\label{secrenyi}

In the sequel we present an entropy invariance statement by change of initial distribution. In particular, we will use this result in the example of the stochastic scrabble (Section \ref{secscrabble}) to compute the entropy of the pushforward measure.
\begin{Th}\label{entropy}
Let $(X_n)_{n \in \mathds{N}}$ be a Markov chain in a finite alphabet $\chi$, with irreducible and aperiodic transition matrix $P=[(p_{ij})]$ and stationary measure $\mu$. For any Markov measure $\nu$ with initial distribution $\pi$ and transition matrix $P$ it holds
$$
H_2(\nu) = H_2(\mu)=-\log p 
$$
where $p$ is the largest eigenvalue of the matrix $[(p_{ij})^2]$.
\end{Th}

\begin{proof} First of all, we observe that the second equality is a well-known result (see e.g. \cite{haydnvaienti} Section 2.2). For the first equality, we will show {the following two inequalities
\begin{equation}\label{ineqhsup}
\overline{H}_2(\nu) \leq H_2(\mu) 
\end{equation}
and
\begin{equation}\label{ineqhinf}  \underline{H}_2(\nu) \geq H_2(\mu).
\end{equation}
}

For convenience here, we will adopt the following notation for strings of stochastic processes: $\{X_n^m=x_n^m\} = \{X_n=x_n,X_{n+1}=x_{n+1}, \cdots,X_m=x_m\} $ for every non-negative integers $n,m$ such that $n\leq m$ and for any realization $x=x_0^{\infty}$.

We will use corollary (3.13) from \cite{FeGa}, which states that there exists $\gamma \in (0,1)$ such that for all $k>1$
$$
\sup_{x_k\in\chi} |\nu(X_k=x_k)- \mu(x_k)| \leq \gamma^k \ .
$$

A straightforward computation gives that for every $n>k>1$ 
$$
\sup_{x_0,x_k\in\chi} |\nu(X_k=x_k\mid X_0=x_0)- \mu(x_k)| \leq \gamma^k $$
and for every $x^n_k\in\chi^{n-k+1}$
$$ |\nu(X^n_k=x^n_k)- \mu(x^n_k)| \leq c\gamma^k \mu(x^n_k) 
$$
with $c=({\inf_{x_0}\{\mu(x_0)\}})^{-1}<+\infty$. 

Let $(a_n)_{n\in \mathds{N}}$ be a non-decreasing and unbounded sequence in $n$ taking values on the non-negative integers and such that $n \geq a_n = o(n)$. Without loss of generality we will only consider the strings $x_0^n$ such that $\nu(X_0^n=x_0^n)>0$. On the one hand, we get
\begin{eqnarray*}
\nu (X_0^n = x_0^n  ) & \leq & \nu (X_{{a_n}}^n=x_{{a_n}}^n  ) \\
                                   & \leq &\left[c\gamma^{a_n}\mu(x_{a_n}^n)+ \mu(x_{a_n}^n)  \right]. \\
                 \end{eqnarray*}

Therefore
\begin{eqnarray*}
\frac{1}{n} \log \sum_{x_0^n} \nu(X_0^n=x_0^n)^2 & \leq & \frac{2}{n} \log (c\gamma^{a_n}+1) + \frac{1}{n} \log \sum_{x_0^n} \mu(x_{a_n}^n)^2 \\
                                                 & = & \frac{2}{n} \log (c\gamma^{a_n}+1) + \frac{1}{n} \log \sum_{x_0^{a_n-1}}\sum_{x_{a_n}^{n}} \mu(x_{a_n}^n)^2 \\
                                                 & \leq & \frac{2}{n} \log (c\gamma^{a_n}+1) + \frac{1}{n} \log |\chi|^{a_n} + \frac{1}{n}\log\sum_{x_{a_n}^{n}} \mu(x_{a_n}^n)^2.
                                                  \end{eqnarray*}

One can observe that the two first terms in the last line vanish as $n\rightarrow\infty$. Moreover, by stationarity of $\mu$ we obtain 
\[\overline{\lim}_{n\rightarrow\infty}\frac{1}{n} \log \sum_{x_0^n} \nu(X_0^n=x_0^n)^2  \leq \overline\lim_{n\rightarrow\infty}\frac{1}{n}\log\sum_{x_{a_n}^{n}} \mu(x_{a_n}^n)^2=\lim_{n\rightarrow\infty}\frac{1}{n-a_n}\log\sum_{x_{0}^{n-a_n}} \mu(x_{0}^{n-a_n})^2=H_2(\mu)\]
which gives us \eqref{ineqhsup}. 

On the other hand, first notice that for strings such that $\nu(X_0^n=x_0^n)>0$, we have for $n$ large enough
\begin{eqnarray*}
\nu(X_0^n=x_0^n) & = & \pi(x_0)P_{x_0x_1}\cdots P_{x_{a_n-1}x_{a_n}}P_{x_{a_n}x_{a_n+1}}\cdots P_{x_{n-1}x_n} \\
                   & \geq & \pi(x_0) \rho^{a_n} \frac{1}{\nu(X_{a_n}=x_{a_n})}\nu(X_{a_n}^n = x_{a_n}^n)   \\    
                   & \geq &  \frac{\pi(x_0) \rho^{a_n}}{\mu(x_{a_n})+\gamma^{a_n}}   \left[\mu(x_{a_n}^n)(1-\gamma^{a_n})\right]    \\
                   & \geq &  d\rho^{a_n}   \left[\mu(x_{a_n}^n)(1-\gamma^{a_n})\right]    
\end{eqnarray*}
where $\rho := \displaystyle\min_{P_{ij}>0}P_{ij}$ and $d=\displaystyle\frac{1}{2}\min_{\pi(x_0)>0}\pi(x_0)$.

Now
\begin{eqnarray*}
\frac{1}{n}\log \sum_{x_0^n}\nu(X_0^n=x_0^n)^2 & \geq & \frac{2}{n}\log \left(d \rho^{a_n}\right)+ \frac{1}{n}\log\sum_{x_0^n}\left[\mu(x_{a_n}^n)(1-\gamma^{a_n})\right]^2 \\
& \geq & \frac{2}{n}\log \left(d \rho^{a_n}\right)+ \frac{2}{n}\log(1-\gamma^{a_n})+\frac{1}{n}\log\sum_{x_{a_n}^n}\left[\mu(x_{a_n}^n)\right]^2.
\end{eqnarray*}
Taking the limit inferior and observing that the first two terms in the last line vanish and the third one converges to $H_2(\mu)$ as $n$ diverges, we obtain \eqref{ineqhinf}. This last statement concludes the proof.

\end{proof}

\subsection{The zero-inflated contamination model}

Let $\{\xi_i\}_{i \in \mathds{N}}$ be a sequence of i.i.d. random variables taking values on $\{0,1\}$, independently of $\mathds{P}$, and governed by a Bernoulli measure such that $\mu(\xi_i = 1) = 1-\epsilon$, where $\epsilon \in (0,1)$ is the noise parameter. Let $f_{\xi}: \Omega \to\Omega$ be a perturbation given by $f_{\xi}(z)=\{\xi_i z_i\}_{i \in \mathds{N}}$. This defines the zero inflated contamination model (see \cite{CoGaLe,GaMo}). 

Then, if $\underline{H}_2(f_{\xi}*\mathds{P})>0$, for $\mathds{P} \otimes \mathds{P}$-almost every $(z,t) \in \Omega \times \Omega$,
\begin{eqnarray*}
\displaystyle\underset{n \to \infty}{\overline{\lim}}\frac{M_n^{f_{\xi}}(z,t)}{\log n} \leq \frac{2}{\underline{H}_2(f_{\xi}*\mathds{P})}.
\end{eqnarray*}

Moreover, if the system $(\Omega, \mathds{P}, \sigma)$ is $\alpha$-mixing with an exponential decay, for $\mathds{P} \otimes \mathds{P}$-almost every $(z,t) \in \Omega \times \Omega$,
\begin{eqnarray*}
\displaystyle\underset{n \to \infty}{\underline{\lim}}\frac{M_n^{f_{\xi}}(z,t)}{\log n} \geq \frac{2}{\overline{H}_2(f_{\xi}*\mathds{P})}.
\end{eqnarray*}
Indeed, for $k$ large enough $f_{\xi}^{-1}C_k \in \mathcal{F}_0^{m_\epsilon(k)},$ where $ m_{\epsilon}(k)$ is the proportion of $1$'s in the $k$-cylinder $C_k(\xi)$. Let  $\mu^{\otimes \mathds{N}}$ denote the product measure that governs the stochastic process $\{\xi_i\}_{i \in \mathds{N}}$. One can observe that by the law of large numbers $\mu^{\otimes\mathds{N}}$-almost every realization $\xi$ has an $\epsilon$-proportion of zeros, i.e.
$$\lim_{k \to \infty}\frac{m_\epsilon(k)}{k}=1-\epsilon.$$
Thus, for $\mu^{\otimes\mathds{N}}$-almost every $\xi$, there exists $\epsilon_1>0$ such that $m_\epsilon(k)=o(k^{1+\epsilon
_1})$ and thus one can apply Theorem \ref{discrete}.

Moreover, if $\mathds{P}$ is a Bernoulli measure we can explicitly compute the R\'enyi entropy of ${f_{\xi}}*\mathds{P}$. Namely, by using the binomial theorem,  for $k$ large enough we get
\begin{eqnarray*}
\sum_{C_k}[\mathds{P}(f_{\xi}^{-1}C_k)]^2 &  = & \sum_{j=1}^{m_{\epsilon}(k)} {{m_{\epsilon}(k)}\choose{j}} p^{2j}(1-p)^{2(m_{\epsilon}(k)-j)} \\
                          &  = & \left[p^2+(1-p)^2\right]^{m_{\epsilon(k)}}.
\end{eqnarray*}

Therefore the R\'enyi entropy is given by
\begin{eqnarray*}
{H}_2(f_{\xi}*\mathds{P})  & = & - \lim_{k \to \infty}\frac{m_\epsilon(k)}{k}\log(p^2+(1-p)^2) \\
      & = & -(1-\epsilon)\log\lt(p^2+(1-p)^2\rt) \ .
\end{eqnarray*}
We observe that if $\chi = \{a_1,\ldots,a_n\}$ is a finite alphabet and $\mathds{P}(X=a_i)=p_i$, by similar computations (and the multinomial theorem) we obtain
$$
{H}_2(f_{\xi}*\mathds{P})= -(1-\epsilon)\log \lt(\sum_{i} p_i^2\rt) = (1-\epsilon) {H}_2(\mathds{P}) \ .
$$

{
Thus, by Theorem \ref{discrete}, we get that for $\mu^{\mathbb{N}}$-almost every realization of $\{\xi_i\}_{i \in \mathbb{N}}$ it holds
}
$$
\frac{M_n^{f_{\xi}}}{\log n}\underset{n\rightarrow\infty}{\longrightarrow}\frac{2}{(1-\epsilon) {H}_2(\mathds{P})} \ \ \ \ \ \mathds{P} \otimes \mathds{P} \ - \ \mbox{a.s.}
$$

The case $f_{\xi}=Id$ is equivalent to $\epsilon=0$ (no contamination), and if $\epsilon$ is close to $1$ we expect to observe larger values for $M_n^{f_{\xi}}$ (in view of Theorem \ref{discrete}). This can be summarized with the following assertion:  the more contamination, the more coincidences appear between the encoded strings. This is a rather intuitive feature of the string matching problem, which indicates that sequences which had lost much information tends to present more similarity.

\subsection{Highest-scoring matching substring}\label{secscrabble}

In this example we will consider the case in which a shorter match can be better scored than a long one, depending on the symbols that compose the matched strings. For this we assume that each string is scored according to the symbols that compose it. In this sense suppose that each letter $a\in \chi$ is associated to a weight $v(a) \in \mathds{N}^*$. We also denote the score of a string $z_0^{m-1}$ by $V(z_0^{m-1}) =\sum_{j=0}^{m-1} v(z_j)$. If $x$ and $y$ are two realizations of the $\chi$-valued stochastic processes  
$(X_n)_n$ and $(Y_n)_n$,

$$V_n(x,y) = \max\limits_{0 \leq i,j \leq n-m}\lt\{V(z_0^{m-1}): \mbox{there exists} \ 1 \leq m \leq n \ \mbox{such that} \ z_0^{m-1} = x_i^{i+m-1}=y_j^{j+m-1} \rt\}$$
is the $n^{th}$ highest-scoring matching substring \cite{ArMoWa}. The authors also named it stochastic Scrabble, because of the namesake board game. 
For two copies independently generated by the same Markov source $\mathds{P}$ with positive transition probabilities $\lt[p_{ij}\rt]$, they stated the following result: 

\begin{equation}\label{scrabble}
\lim_{n \to \infty} \frac{V_n}{\log n}=\frac{2}{-\log p} \ \ \ \ \ \ \ \ \  \mathds{P} \times \mathds{P} -\mbox{a.s.} \ ,
\end{equation}
where $p \in (0,1)$ is the largest root of $\det(P - \lambda^V)=0$, with $P=\lt[p^2_{ij}\rt]$ and $\lambda^V= \lt[\delta_{ij}\lambda^{v(i)}\rt]$.

One can observe that this result \eqref{scrabble} can be obtained as particular case of Theorem \ref{discrete}. Indeed, inspired by \cite{ArMoWa}, we can construct a specific {encoder} $f$ that stretches the sequences depending on the weights of its letters. Formally
\begin{align}
  f \  \colon & \  \chi^{\mathds{N}}\to \chi^{\mathds{N}} \nonumber \\ \label{old}
 & x_0^{\infty} \mapsto \underbrace{x_0x_0\cdots x_0}_{v(x_0)} \underbrace{x_1x_1\cdots x_1}_{v(x_1)} \cdots \underbrace{x_nx_n\cdots x_n}_{v(x_n)} \cdots
\end{align}

With this particular {encoder}, we get that $M_n^f(x,y)=V_n(x,y)$ and thus to get \eqref{scrabble} we need to compute $H_2(f_*\mathds{P})$ and check that conditions (i) and (ii) of Theorem \ref{discrete} are satisfied. 

We recall that if $(X_n)$ is a Markov chain in $\chi = \{1,2,\cdots,d\}$, we can see $f(X_n)$ as a Markov Chain in $\tilde{\chi}$, which is a $(\sum_{i \in \chi}v(i))$-sized alphabet, given by
$$
\tilde\chi = \left\{1_1,1_2,\cdots,1_{v(1)},2_1,2_2,\cdots,2_{v(2)}, \cdots, d_1, d_2,\cdots,d_{v(d)}\right\} \ .
$$
 In this context, we will consider that $f:\chi^{\mathds{N}} \to \tilde{\chi}^{\mathds{N}}$.
Furthermore, if $Q = [Q_{ij}]$, $1\leq i,j \leq d$ is the transition matrix for $(X_n)$ we get that the transition matrix $Q^*$ for the chain $(f(X_n))$ on $\tilde\chi$ is given by
$$
\begin{array}{lll}
Q_{i_\ell i_{\ell +1}}^* = 1 &  \mbox{if} & 1\leq \ell \leq v(i)-1  \ \mbox{and} \ 1\leq i,j \leq d \ ; \\
Q_{i_{v(i)} j_{1}}^* = Q_{ij} & \mbox{if}  &  1\leq i,j \leq d \ ; \\
Q_{ij}^*=0 & & \mbox{otherwise}.
\end{array}
$$

Notice that, if $v_{min}=\min_{i \in \chi}\{v(i)\}$ is the minimum weight, we get for any cylinder $C_n$,
$$
f^{-1}C_n \in \mathcal{F}_0^{\left\lfloor \frac{n}{v_{min}}\right\rfloor} \ ,
$$
and since $n/v_{min}=o(n^{1+\epsilon})$ for all $\epsilon > 0$, condition (ii) of Theorem \ref{discrete} is then satisfied.
We recall that an irreducible and aperiodic positive recurrent Markov chain is an $\alpha$-mixing process with exponential decay of correlation (see e.g. Theorem 4.9 in \cite{LevinPeresWilmer}) which implies condition (i).

Finally, to obtain \eqref{scrabble}, we need to compute $H_2(f_*\mathds{P})$. As in \cite{ArMoWa}, to assure aperiodicity for the encoded process $f(X_n)$ we assume that $gdc\{v(1), v(2), \ldots, v(d)\}=1$.

Moreover, by Theorem \eqref{entropy} we know that the R\'enyi entropy of its stationary measure $\mu$ is given by $H_2(\mu)=-\log p$, where $p$ is the largest positive eigenvalue of the matrix $\lt[(Q^*)^2_{ij}\rt]$, $1\leq i,j \leq (\sum_{i \in \chi}v(i))$ (it was proved in  \cite{ArMoWa} that this $p$ is the same as the one defined in \eqref{scrabble}). Moreover, we observe that $f_*\mathds{P}$ is a Markov measure with initial distribution $\pi$ and transition matrix $Q^*$, where $\pi$ is defined by
$\pi(i_1)=\mathds{P}(X_0=i)$ and $\pi(i_j)=0$ for any $i\in \chi$ and $1<j\leq v(i)$. It is important to notice that in general, $f_*\mathds{P}$ is not stationary.

Thus, by Theorem \eqref{entropy}, we have $H_2(\mu)=H_2(f_*\mathds{P})$ and we can combine it with equation ({\footnotesize{$\largestar$}}) in Theorem \ref{discrete} to conclude that, for $\mathds{P} \times \mathds{P}$ almost every pair of realizations, as $n$ diverges it holds 
$$
\frac{V_n}{\log n} \longrightarrow\frac{2}{-\log p} \ .
$$
We remark that this example generalizes \cite{ArMoWa} to $\alpha$-mixing processes with exponential decay and $\psi$-mixing with polynomial decay, since we can apply Theorem \ref{discrete} to this {encoder} $f$, and then obtain information on the highest scoring $V_n$.

\section{Shortest distance between observed orbits}\label{shortest}

In \cite{BaLiRo} it was explained that, in the case of dynamical systems, investigating the longest common substring is similar to the study of the shortest distance between orbits. Mixing this idea with the fact that studying statistical properties of observations of dynamical systems could be more significant for experimentalists (see e.g \cite{Bosh,RS,Rousseau}), we will analyse in this section the behaviour of the shortest distance between two observed orbits.

Let $(X, \mathcal{A}, T, \mu)$ be a dynamical system  where $(X,d)$ is a metric space, $\mathcal{A}$ is a $\sigma$-algebra on $X$, $T:X\rightarrow X$ is a
measurable map and $\mu$ an invariant probability measure on $(X, \mathcal{A})$ i.e., $\mu(T^{-1}(A))=\mu(A),$ for all $A \in \mathcal{A}.$

\begin{Df}\label{shortestdist}
Let $f:X\rightarrow Y \subset \mathds{R}^N$ be a measurable function, called the observation. We define the shortest distance between two observed orbits as follows
$$m_n^f(x,y) = \min_{i,j=0,\ldots,n-1}\left(d(f(T^ix),f(T^jy))\right).$$
\end{Df}

For a measure $\nu$ on $X$ we define \textit{the lower and upper correlation dimension} of $\nu$ by
$$\underline{C}_{\nu} = \underset{r \rightarrow 0}{\underline{\lim}}\dfrac{\log \int_X \nu(B(x,r))\ d\nu(x)}{\log r} \ \ \mbox{and} \ \ \overline{C}_{\nu} =  \underset{r \rightarrow 0}{\overline{\lim}}\dfrac{\log \int_X \nu(B(x,r))\ d\nu(x)}{\log r}.$$
If the limit exists, we denote by $C_\nu$ the common value.

We will show that the shortest distance between two observed orbits is related with the correlation dimension of the pushforward measure $f_*\mu$. Recall that the pushforward measure is given by $f_*\mu(\cdot):= \mu(f^{-1}(\cdot))$.

\begin{Th}\label{firstresult}Let $(X, \mathcal{A}, T, \mu)$ be a dynamical system. Consider an observation $f:X \to Y$ such that $\underline{C}_{f_*\mu}>0$. Then for $\mu \otimes \mu$-almost every $(x,y) \in X \times X$
\begin{equation}\label{ineqsd}
\underset{n \rightarrow \infty}{\overline{\lim}}\frac{\log m_n^f(x,y)}{-\log n} \leq \frac{2}{\underline{C}_{f_*\mu}} \ .
\end{equation}
\end{Th}

We recall that the condition $\underline{C}_{f_*\mu} = 0$ can lead to unknown values for the above limit. However, one can observe that if $m_n^f=0$ on a set of positive measure, our result implies immediately that $\underline{C}_{f_*\mu} = 0$. The following simple example illustrates this fact. 
\begin{Ex}
Let $X=[0,1]$ and  $\mu=Leb$ the Lebesgue measure on $X$. Given $ A \subset X$ with $\mu(A)>0$ we define a function $f:X \to X$ by
$$f(x)=\left\{\begin{array}{ll}
x,\text{ if } x\in A^c \\
c, \text{ if } x \in A \\
\end{array}\right.$$
where $c\in[0,1]$ is a constant. 
For any transformation $T$ which is $\mu$-invariant, we have $m_n^f(x,y)=0$ for every $x,y\in A$, and thus $\underline{C}_{f_*\mu}=0$. {One can also observe that if T is ergodic, for $n$ sufficiently large $m_n^f(x,y)=0$ for almost every $x,y$. Indeed, by Poincar\'e recurrence Theorem, we obtain that, for almost every $x,y$, the orbits of $x$ and $y$ will visit $A$, i.e. it exist $n_1,n_2\in \N$ such that $T^{n_1}(x)\in A$ and $T^{n_2}(y)\in A$. Therefore, for $n$ sufficiently large $m_n^f(x,y)=d(f(T^{n_1}(x)), f(T^{n_2}(y)))=0$.}
In fact, with a simple computation, one can show that ${C}_{f_*\mu}=0$.
\end{Ex}

\begin{proof}[Proof of Theorem \ref{firstresult}]
For $\epsilon>0$ we define
$$k_n = \frac{2\log n + \log \log n}{\underline{C}_{f_*\mu}-\epsilon}.$$
We also define
$$A_{ij}^f(y) = T^{-i}\left[f^{-1}B\lt(f(T^jy), e^{-k_n}\rt)\right] $$
and
$$S_n^f (x,y)= \displaystyle\sum_{i,j=1,\ldots,n} \mathbbm{1}_{A_{ij}^f(y)}(x).$$

Using Markov inequality, we get that
$$\mu \otimes \mu \lt(\lt\{(x,y): m_n^f(x,y) < e^{-k_n}\rt\}\rt)=\mu \otimes \mu \lt(\lt\{(x,y): S_n^f(x,y)>0\rt\}\rt)  \leq  \E\lt(S_n^f\rt).$$
Using the invariance of $\mu$, we can compute the expected value of $S_n^f$ 
\begin{eqnarray*}
\E\lt(S_n^f\rt) 
                                                 & =& n^2 \int f_*\mu\left(B\lt(f(y), e^{-k_n}\rt)\right) \ d\mu(y).
\end{eqnarray*}
Thus, for $n$ large enough, by definition of $\underline{C}_{f_*\mu}$ and $k_n$, we obtain
$$
\mu \otimes \mu \lt(\lt\{(x,y): m_n^f(x,y) \leq e^{-k_n}\rt\}\rt) \leq n^2 e^{-k_n (\underline{C}_{f_*\mu}-\epsilon)} = \frac{1}{\log n} \ .
$$

Choosing a subsequence $\{n_{\kappa}\}_{\kappa \in \mathds{N}}$ such that $n_{\kappa}= \lceil e^{\kappa^2}\rceil$, we can use Borel-Cantelli Lemma as in the proof of \eqref{ineq1discretecase} to obtain
$$
\underset{n \rightarrow \infty}{\overline{\lim}} \frac{\log m_n^f(x,y)}{-\log n} \leq \frac{2}{\underline{C}_{f_*\mu}-\epsilon}.
$$

Since $\epsilon$ can be arbitrarily small, the proof is complete.
\end{proof}

As in \cite{BaLiRo}, to obtain an equality in \eqref{ineqsd}, we will need more assumptions on the system.

\begin{description}
\item [(H1)] Let $\mathcal{H}^\alpha(X,\mathds{R})$  be the space of H\"older observables. For all $\psi, \phi \in \mathcal{H}^\alpha(X,\mathds{R})$ and for all $n \in \mathds{N}^{*},$ we have:
$$\lt|\int_X \psi \circ f(T^n x)\phi\circ f(x) \ d\mu(x) - \int_X \psi\circ f \ d\mu \int_X \phi \circ f \ d\mu \rt| \leq \|\psi \circ f\|_{\alpha} \|\phi \circ f\|_{\alpha} \theta_n$$
with $\theta_n=a^{n}$ and $a\in[0,1)$.

\item [(HA)] There exist $r_0>0, \ \xi \geq 0$ and $\beta>0$ such that for $f_*\mu$-almost every $y \in \mathds{R}^N$ and any $r_0>r> \rho>0,$
$$f_{\ast}\mu(B(y,r+\rho)\backslash B(y,r-\rho)) \leq r^{-\xi}\rho^{\beta}.$$
\end{description}

One can observe that, if $f$ is Lipschitz, assuming hypothesis (H1) is weaker than assuming a exponential decay of correlations (for H\"older observables) for the system $(X, \mathcal{A}, T, \mu)$. Indeed, note that if $f$ is Lipschitz then $\psi \circ f$ is H\"older for every H\"older function $\psi$.

\begin{Th}\label{secondresult}Let $(X, \mathcal{A}, T, \mu)$ be a dynamical system and consider a Lipschitz observation $f:X \to Y$ such that $\underline{C}_{f_*\mu}>0$. If the system satisfies (H1) and (HA) then for $\mu \otimes \mu$-almost every $(x,y) \in X \times X$

$$
\underset{n \rightarrow \infty}{\underline{\lim}}\frac{\log m_n^f(x,y)}{-\log n} \geq \frac{2}{\overline{C}_{f_*\mu}} \ .
$$

Furthermore, if ${C}_{f_*\mu}$ exists, we get

\begin{equation}\tag{\footnotesize{$\largestar \largestar$}}
\lim_{n \to \infty}\frac{\log m_n^f(x,y)}{-\log n} = \frac{2}{C_{f_*\mu}}
\end{equation}
 for $\mu \otimes \mu$-almost every $(x,y) \in X \times X$.
\end{Th}

In what follows one will observe that the proof follows the lines of the symbolic case where $M_n^f$ will be substitute by $-\log m_n^f$ and cylinders of size $k$ will be substitute by balls of radius $e^{-k}$. Thus, we will only write the main lines of the proof, giving more details when the proof diverge from the symbolic one.

To prove Theorem \ref{secondresult}, the main difficulty and difference with the symbolic case is that we cannot apply mixing as simply. In particular, we can only apply mixing to H\"older observables and indicator functions are not even continuous. To overthrow this difficulty, we will first prove in the following lemma that a particular function is H\"older. In the proof of Theorem \ref{secondresult}, we will apply the mixing property to this particular function.

\begin{Lem}\label{HA} Let $(X, \mathcal{A}, T, \mu)$ be a dynamical system with observation $f$. If it satisfies $(HA)$, then there exist $0< r_0<1, \ c \geq 0$ and $\zeta \geq 0$ such that for any $0<r<r_0$, the function $\psi_1:  x \mapsto f_*\mu (B(x,r))$ belongs to $\mathcal{H}^{\alpha}(X,\mathds{R})$ and
$$
||\psi_1||_{\alpha} \leq 2r^{-\zeta} \ .
$$
\end{Lem}

\begin{proof}
Let $x,y \in X$ and $0<r<r_0$, if $||x-y||<r$ we have
$$
||f_*\mu (B(y,r))-f_*\mu (B(x,r))|| \leq f_*\mu (B(x,r+||x-y||)) - f_*\mu (B(x,r-||x-y||)).
$$
Thus, by (HA), 
$$
||f_*\mu (B(y,r))-f_*\mu (B(x,r))|| \leq r^{-\xi}||x-y||^\beta.
$$
On the other hand, if $||x-y|| \geq r$ then
$$
||f_*\mu (B(y,r))-f_*\mu (B(x,r))|| \leq 2 \leq \frac{2}{r} ||x-y|| \ .
$$
Thus, $\psi_1$ is H\"older and $||\psi_1||_{\alpha} \leq 2r^{-\zeta}$ with $\zeta=\max\{1,\xi\}$.

\end{proof}

In the sequel, we present the proof of Theorem \ref{secondresult}. This proof mainly follows the ideas of the proof of \cite[Theorem 5]{BaLiRo}.

\begin{proof}[Proof of Theorem \ref{secondresult}.]
Without loss of generality, we will assume here that $\theta_n=e^{-n}$. Let $b<-4$. Given $\epsilon >0$, we define
$$k_n = \frac{2 \log n+ b \log\log n}{\overline{C}_{f_*\mu}+\epsilon}.$$

By Chebyshev's inequality we get that 
$$\mu \otimes \mu \lt(\lt\{(x,y): m_n^f(x,y) \geq e^{-k_n}\rt\}\rt) \leq  \frac{\textrm{var} \lt(S_n^f\rt)}{\E\lt(S_n^f\rt)^2}.$$
We now proceed to estimate the variance of $S_n^f.$

We see at once that
\begin{eqnarray}\label{var}
\textrm{var}\lt(S_n^f\rt) &  =  & \sum_{1 \leq i,i',j,j' \leq n} \textrm{\textrm{cov}}\left(\mathbbm{1}_{A_{ij}^f},\mathbbm{1}_{A_{i'j'}^f}\right) \nonumber\\
         &=& \sum_{1 \leq i, i', j, j' \leq n} \int \int \mathbbm{1}_{f^{-1}B\lt(f(T^jy),e^{-k_n}\rt)}(T^ix) \mathbbm{1}_{f^{-1}B\lt(f(T^{j'}y),e^{-k_n}\rt)}(T^{i'}x) \nonumber \\
& & -n^4 \lt(\int f_{\ast}\mu\lt(B\lt(f(y),e^{-k_n}\rt)\rt)\ d\mu(y)\rt)^2.
\end{eqnarray}
{One can observe that this equation is similar to \eqref{vardis}, thus as in the symbolic case we would like to apply the mixing property to estimate the previous sum. However, in this case our assumption (H1) only allows us to use mixing with H\"older functions
thus we will approximate our characteristic functions by Lipschitz (and thus H\"older) functions following the construction of the proof of Lemma 9 in \cite{Saussol1}.}

Let $\rho>0$ (to de defined properly later). Let $\eta_{e^{-k_n}}:[0, \infty) \to \mathds{R}$ be the $\frac{1}{{\rho e^{-k_n}}}$-Lipschitz function such that $\mathbbm{1}_{[0,{e^{-k_n}}]} \leq \eta_{e^{-k_n}} \leq \mathbbm{1}_{[0,(1+\rho){e^{-k_n}}]}$ and set $\varphi_{f(y),{e^{-k_n}}}(x)= \eta_{e^{-k_n}}(d(f(y),x))$. Since $f$ is $L$-Lipschitz it follows that $\varphi_{f(y),{e^{-k_n}}} \circ f$ is $\frac{L}{{\rho e^{-k_n}}}$-Lipschitz. Moreover, we have
\begin{eqnarray}\label{eqapprox}
\mathbbm{1}_{f^{-1}B\lt(f(T^jy),e^{-k_n}\rt)}(x) &=& \mathbbm{1}_{B\lt(f(T^jy),e^{-k_n}\rt)}(f(x)) \nonumber\\
&=& \mathbbm{1}_{[0,e^{-k_n}]}(d(f(T^jy),f(x)))\nonumber \\
& \leq & \eta_{e^{-k_n}}(d(f(T^jy),f(x)))\nonumber \\
&=&\varphi_{f(T^jy),e^{-k_n}}(f(x)).
\end{eqnarray}

We are now able to apply the mixing property and as in the symbolic case, we will consider four different cases. Let us fix $g=g(n)=\log(n^\gamma)$ for some $\gamma>0$ to be defined later.

Case 1: $|i-i'| > g$. By (H1) and \eqref{eqapprox} we obtain
\begin{eqnarray*}
&&\int \int \mathds{1}_{f^{-1}B\lt(f(T^jy),e^{-k_n}\rt)}(T^{i}x) \mathds{1}_{f^{-1}B\lt(f(T^{j'}y),e^{-k_n}\rt)}(T^{i'}x) \ d\mu(x) \ d\mu(y) \\
&=& \int \int \mathds{1}_{f^{-1}B\lt(f(T^jy),e^{-k_n}\rt)}(T^{i-i'}x) \mathds{1}_{f^{-1}B\lt(f(T^{j'}y),e^{-k_n}\rt)}(x) \ d\mu(x) \ d\mu(y) \\
&\leq &\int \int \varphi_{f(T^jy),e^{-k_n}}(f(T^{i-i'}x))\varphi_{f(T^{j'}y),e^{-k_n}}(f(x)) \ d\mu(x)\ d\mu(y) \\
& \leq & \int  \lt(\int \varphi_{f(T^jy),e^{-k_n}}(f(T^{i-i'}x))\  d\mu(x) \int\varphi_{f(T^{j'}y),e^{-k_n}}(f(x)) \ d\mu(x)  \rt)\ d\mu(y) \\
& +& \theta_g \lt\Vert\varphi_{f(T^jy),e^{-k_n}}\rt\Vert \lt\Vert \varphi_{f(T^{j'}y),e^{-k_n}} \rt\Vert \\
& \leq & \frac{L^2}{\rho^2 e^{-2k_n}}\theta_g + \int f_*\mu\lt(B\lt(f(T^jy),(1+\rho)e^{-k_n}\rt)\rt) f_*\mu\lt(B\lt(f(T^{j'}y),(1+\rho)e^{-k_n}\rt)\rt) \ d\mu(y).
\end{eqnarray*}
{This estimate is similar to \eqref{eq1}, however one needs to take extra care to estimate the second part (and not use mixing immediately as in the discrete case) since the radius of the balls is not $e^{-k_n}$ anymore. Indeed, we need the radius to be $e^{-k_n}$ so that when we will use mixing again we will obtain a term which will simplify with the last term in \eqref{vardis}.}
To do so, we can observe that using (HA) we obtain
\begin{eqnarray*}
&&\int f_*\mu\lt(B\lt(f(T^jy),(1+\rho)e^{-k_n}\rt)\rt) f_*\mu\lt(B\lt(f(T^{j'}y),(1+\rho)e^{-k_n}\rt)\rt) \ d\mu(y) \\
& &- \int f_*\mu\lt(B\lt(f(T^jy),e^{-k_n}\rt)\rt) f_*\mu\lt(B\lt(f(T^{j'}y),e^{-k_n}\rt)\rt) \ d\mu(y) \\
& \leq & \int f_*\mu\lt(B\lt(f(T^jy),(1+\rho)e^{-k_n}\rt)\rt)\lt( f_*\mu\lt(B\lt(f(T^{j'}y),(1+\rho)e^{-k_n}\rt)\rt)-f_*\mu\lt(B\lt(f(T^{j'}y),e^{-k_n}\rt)\rt) \rt) \ d\mu(y) \\
& &+ \int f_*\mu\lt(B\lt(f(T^{j'}y),e^{-k_n}\rt)\rt)\lt( f_*\mu\lt(B\lt(f(T^jy),(1+\rho)e^{-k_n}\rt)\rt)- f_*\mu\lt(B\lt(f(T^{j}y),e^{-k_n}\rt)\rt)\rt) \ d\mu(y) \\
& \leq & \int f_*\mu\lt(B\lt(f(T^jy),(1+\rho)e^{-k_n}\rt)\rt)e^{\xi k_n}\rho^{\beta} \ d\mu(y) +  \int f_*\mu\lt(B\lt(f(T^{j'}y),e^{-k_n}\rt)\rt)e^{\xi k_n}\rho^{\beta} \ d\mu(y).
\end{eqnarray*}
Therefore, choosing $\rho=n^{-\delta}$ for some $\delta>0$ to be defined later, we have for $n$ large enough

\begin{eqnarray*}
&&\int \int \mathds{1}_{f^{-1}B\lt(f(T^jy),e^{-k_n}\rt)}(T^{i}x) \mathds{1}_{f^{-1}B\lt(f(T^{j'}y),e^{-k_n}\rt)}(T^{i'}x) \ d\mu(x) \ d\mu(y) \\
& \leq & \frac{L^2}{\rho^2 e^{-2k_n}}\theta_g+ 2e^{\xi k_n}\rho^{\beta} \int f_*\mu\lt(B\lt(f(T^jy),2e^{-k_n}\rt)\rt) \ d\mu(y) \\
& &+ \int f_*\mu\lt(B\lt(f(T^jy),e^{-k_n}\rt)\rt) f_*\mu\lt(B\lt(f(T^{j'}y),e^{-k_n}\rt)\rt) \ d\mu(y) .\\
\end{eqnarray*} 
{One can observe that in this estimate we have now an additional term that was not present in the symbolic setting \eqref{eq1} which is due to the need to approximate characteristic functions by Lipschitz functions.}
To deal with the third term of the last inequality we need to consider two different cases.

Case 1.1: $\vert j- j^{'}\vert>g.$ We can use the mixing property (H1) to the particular function defined in Lemma \ref{HA}
\begin{eqnarray*}
& & \int  f_*\mu\lt(B\lt(f(T^jy),e^{-k_n}\rt)\rt)  f_*\mu\lt(B\lt(f(T^{j'}y),e^{-k_n}\rt)\rt) \ d\mu(y) \\
&\leq& 4\theta_g e^{2\zeta k_n}+\left(\int  f_*\mu\lt(B\lt(f(y),e^{-k_n}\rt)\rt)\ d\mu(y)\right)^2 \\
\end{eqnarray*}
and we obtain an estimate similar to \eqref{eq11}.

Case 1.2: $\vert j- j^{'}\vert \leq g.$ Using Holder's inequality together and the invariance of $\mu$, as in \eqref{eq12}, we have
\begin{eqnarray*}
& & \int  f_*\mu\lt(B\lt(f(T^jy),e^{-k_n}\rt)\rt)  f_*\mu\lt(B\lt(f(T^{j'}y),e^{-k_n}\rt)\rt) \ d\mu(y) \\
&\leq& \int  f_*\mu\lt(B\lt(f(y),e^{-k_n}\rt)\rt)^2 \ d\mu(y).
\end{eqnarray*}

Case 2.1: $|i-i'| \leq  g$ and $\vert j- j^{'}\vert>g.$ In this case, we obtain the same estimate as in the case 1.2 using the following symmetry: 
$$
\mathds{1}_{f^{-1}B\lt(f(T^{\ell}y),e^{-k_n}\rt)}(T^{m}x) = \mathds{1}_{f^{-1}B\lt(f(T^mx),e^{-k_n}\rt)}(T^{\ell}y)
$$
for all $\ell, m \in \mathds{N}$ and all $x$ and $y$.

Case 2.2: $|i-i'| \leq  g$ and $\vert j- j^{'}\vert \leq g.$ The boundedness of the indicator function and invariance of $\mu$ give that,
\begin{eqnarray*}
& & \int \int \mathds{1}_{f^{-1}B\lt(f(T^jy),e^{-k_n}\rt)}(T^{i}x) \mathds{1}_{f^{-1}B\lt(f(T^{j'}y),e^{-k_n}\rt)}(T^{i'}x) \ d\mu(x) \ d\mu(y)  \\
& \leq & \int \int \mathds{1}_{f^{-1}B\lt(f(T^jy),e^{-k_n}\rt)}(T^{i}x) \ d\mu(x) \ d\mu(y)
=
 \int f_*\mu \lt(B\lt(f(y),e^{-k_n}\rt)\rt) \ d\mu(y).
\end{eqnarray*}

Putting all the previous estimates in \eqref{var} we obtain

\begin{eqnarray}\label{var/exp}
\frac{\textrm{var}\lt(S_n^f\rt)}{\mathbb{E}\lt(S_n^f\rt)^2} & \leq & \frac{n^4L^2\rho^{-2}e^{2 k_n}\theta_g +4n^4\theta_ge^{\zeta k_n}+ 2n^4e^{ \xi k_n}\rho^{\beta} \int f_*\mu\lt(B\lt(f(y),2e^{-k_n}\rt)\rt) \ d\mu(y) }{\lt(n^2 \int f_*\mu\lt(B\lt(f(y),e^{-k_n}\rt)\rt) \ d\mu(y)\rt)^2} \nonumber \\
&+& \frac{ 4n^2g^2 \int f_*\mu\lt(B\lt(f(y),e^{-k_n}\rt)\rt)\ d\mu(y) +4n^3g \int f_*\mu\lt(B\lt(f(y),e^{-k_n}\rt)\rt)^2 \ d\mu(y)}{\lt(n^2 \int f_*\mu\lt(B\lt(f(y),e^{-k_n}\rt)\rt) \ d\mu(y)\rt)^2}. \nonumber \\
\end{eqnarray}

{This estimate is comparable to \eqref{pxp} (in the symbolic setting), except for the third term coming from our approximation of characteristic functions, and the terms will be dealt with in a similar way. To help the reader understanding the following majorations, we can observe that
\[ n^{-2}(\log n)^{-b}=e^{-k_n(\overline{C}_{f_*\mu}+\epsilon)}\lesssim\int f_*\mu\lt(B\lt(f(y),e^{-k_n}\rt)\rt) \ d\mu(y) \leq 1.\]
}
Recalling that $\rho=n^{-\delta}$, we can choose $\delta$ large enough (depending on $\xi, \beta, \overline{C}_{f_*\mu}, b$ and $\epsilon$) so that

\begin{equation}\label{term2}
\frac{ 2n^4e^{ \xi k_n}\rho^{\beta} \int f_*\mu\lt(B\lt(f(y),2e^{-k_n}\rt)\rt) \ d\mu(y) }{\lt(n^2 \int f_*\mu\lt(B\lt(f(y),e^{-k_n}\rt)\rt) \ d\mu(y)\rt)^2} \leq  \frac{1}{n}.
\end{equation}

Recalling that $g=\log(n^{\gamma})$, we can observe that we can choose $\gamma$ large enough (depending on $\delta, \overline{C}_{f_*\mu},\zeta,b$ and $\epsilon$) so that

\begin{equation}\label{term1}
\frac{n^4L^2\rho^{-2}e^{2 k_n}\theta_g}{\lt(n^2 \int f_*\mu\lt(B\lt(f(y),e^{-k_n}\rt)\rt)\ d\mu(y)\rt)^2} \leq \frac{1}{n} 
\end{equation}
and so that
\begin{equation}\label{term3}
\frac{4n^4\theta_ge^{\zeta k_n}}{\lt(n^2 \int f_*\mu\lt(B\lt(f(y),e^{-k_n}\rt)\rt)\ d\mu(y)\rt)^2} \leq \frac{1}{n} .
\end{equation}

For the fourth term we have
\begin{eqnarray}\label{term4}
\frac{4n^2g^2 \int f_*\mu\lt(B\lt(f(y),e^{-k_n}\rt)\rt) \ d\mu(y) }{\lt(n^2 \int f_*\mu\lt(B\lt(f(y),e^{-k_n}\rt)\rt)\ d\mu(y)\rt)^2} 
& \leq & \frac{4g^2}{n^2e^{-k_n(\overline{C}_{f_*\mu}+ \epsilon)}} \nonumber \\
& \leq & 4\gamma^2(\log n)^{2+b}.
\end{eqnarray}
To estimate the last term, one cannot use immediately the subadditivity as in \eqref{eq12}, thus we will use the following lemma.
\begin{Lem}[Lemma 14 \cite{BaLiRo}]
Let $Z\subset \mathbb{R}^N$ and let $\nu$ be a probability measure on $Z$. There exists a constant $K>0$ depending only on $N$ such that for every $r$ small enough
\[\int_Z\mu\left(B(y,r)\right)^2d\nu(y) \leq K\left(\int_Z \mu \left(B(y,r)\right)d\nu(y)\right)^{3/2}.\]
\end{Lem}
Applying the previous lemma with $Z=Y$ and  $\nu=f_*\mu$ we obtain

\begin{eqnarray}\label{term5}
\frac{4n^3g \int f_*\mu\lt(B\lt(f(y),e^{-k_n}\rt)\rt)^2 \ d\mu(y)}{\lt(n^2\int f_*\mu\lt(B\lt(f(y),e^{-k_n}\rt)\rt) \ d\mu(y)\rt)^2} 
& \leq & \frac{4g K }{n\lt(\int f_*\mu(B(f(y),e^{-k_n})) \ d\mu(y)\rt)^{1/2} } \nonumber \\
& \leq & \frac{4g K }{n}e^{\frac{k_n(\overline{C}_{f_*\mu}+ \epsilon)}{2}} \nonumber \\
& \leq & 4K \gamma (\log n)^{1+ \frac{b}{2}}.
\end{eqnarray}

Since $b<-4$ and substituing \eqref{term1}, \eqref{term2}, \eqref{term3},  \eqref{term4} and \eqref{term5} into \eqref{var/exp} we get
\begin{eqnarray*}
\mu \otimes \mu \lt(\lt\{(x,y): m_n^f(x,y) \geq e^{-k_n}\rt\}\rt) \leq   
\mathcal{O}((\log n)^{-1}).
\end{eqnarray*}

Thus, taking a subsequence $n_{\kappa}= \lceil e^{\kappa^2}\rceil.$ As in the proof of Theorem \ref{firstresult}, by Borel-Cantelli Lemma we obtain
$$
\underset{n \rightarrow \infty}{\underline{\lim}} \frac{\log m_n^f(x,y)}{-\log n} = \underset{\kappa \rightarrow \infty}{\underline{\lim}}\frac{\log m_{n_{\kappa}}^f(x,y)}{-\log n_{\kappa}} \geq \frac{2}{\underline{C}_{f_*\mu}+\epsilon}.
$$

Since $\epsilon$ can be arbitrarily small, the theorem follows.

\end{proof}

Following the idea of \cite{RS,Rousseau} that the study of observation of dynamical systems can be used to study random dynamical systems, we will show in the next section that the previous result can be applied to obtain information on the shortest distance between two random orbits.

\section{Shortest distance between orbits for random dynamical systems}\label{random}

Let $X\subset \mathbb{R}^N$ and let $(\Omega, \theta, \mathds{P})$ be a probability measure preserving system, where $\Omega$ is a metric space and $B(\Omega)$ its Borelian $\sigma$-algebra. We first introduce the notion of random dynamical system.

\begin{Df}
A random dynamical system $\mathcal{T}=(T_{\omega})_{\omega \in \Omega}$ on $X$ over $(\Omega, B(\Omega), \mathds{P}, \theta)$ is generated by maps $T_\omega$ such that $(\omega,x) \mapsto T_\omega(x)$ is measurable and satisfies:
$$T_\omega^0=Id \ \mbox{for all} \ \omega \in \Omega,$$
$$T_\omega^n= T_{\theta^{n-1}(\omega)}\circ \cdots \circ T_{\theta(\omega)} \circ T_\omega \ \mbox{for all} \ n \geq 1.$$
The map $S: \Omega \times X \to \Omega \times X$ defined by $S(\omega, x)=(\theta(\omega),T_\omega(x))$ is the dynamics of the random dynamical systems generated by $\mathcal{T}$ and is called skew-product.
\end{Df}
\begin{Df}
A probability measure $\mu$ is said to be an invariant measure for the random dynamical system $\mathcal{T}$ if it satisfies

\begin{itemize}
\item [1.] $\mu$ is $S$-invariant
\item [2.] $\pi_*\mu= \mathds{P}$
\end{itemize}
where $\pi: \Omega \times X \to \Omega$ is the canonical projection.
\end{Df}

Let $(\mu_\omega)_\omega$ denote the decomposition of $\mu$ on $X$, that is, $d\mu(\omega,x)=d\mu_\omega(x)d\mathds{P}(\omega)$. We denote by $\nu=\int \mu_\omega d\mathds{P}$ the marginal of $\nu$ on $X$.
\begin{Df}
We define the shortest distance between two random orbits by
$$m_n^{\omega, \tilde{\omega}}(x,\tilde{x}) =  \min_{i,j=0, \ldots, n-1} \lt(d\lt(T_\omega^i(x),T_{\tilde{\omega}}^j(\tilde{x})\rt)\rt).$$
\end{Df}
As in the deterministic case, we need a hypothesis for the measure and an (annealed) exponential decay of correlations for the random dynamical system. Namely,
\begin{description}
  \item[(a)] There exist $r_0>0, \ \xi \leq 0$ and $\beta>0$ such that for almost every $y \in X$ and any $r_0>r> \rho>0,$
$$\nu(B(y,r+\rho)\backslash B(y,r-\rho)) \leq r^{-\xi}\rho^{\beta}.$$

\item[(b)] (Annealed decay of correlations) $\forall n \in \mathds{N}^*$, $\psi$ and $\phi$ H\"older observables from $X$ to $\mathds{R}$,
$$\lt|\int_{\Omega \times X} \psi (T^n_\omega(x))\phi(x) \ d\mu(\omega,x) - \int_{\Omega \times X} \psi \ d\mu \int_{\Omega \times X} \phi \ d\mu \rt| \leq \|\psi \|_{\alpha}\|\phi \|_{\alpha} \theta_n$$
with $\theta_n=e^{-n}$.
\end{description}

\begin{Th}\label{theoremrandom}
Let $\mathcal{T}$ be a random dynamical system on $X$ over $(\Omega, B(\Omega), \mathds{P}, \theta)$ with an invariant measure $\mu$ such that $\underline{C}_{\nu}>0$. Then for $\mu \otimes \mu$-almost every $(\omega, x,\tilde{\omega},\tilde{x}) \in \Omega \times X \times \Omega \times X,$
$$
\underset{n \rightarrow \infty}{\overline{\lim}}\frac{\log m_n^{\omega, \tilde{\omega}}(x,\tilde{x})}{-\log n} \leq \frac{2}{\underline{C}_{\nu}}. \
$$
Moreover, if the random dynamical system satisfies assumptions $(a)$ and $(b)$, then
$$
\underset{n \rightarrow \infty}{\underline{\lim}}\frac{\log m_n^{\omega, \tilde{\omega}}(x,\tilde{x})}{-\log n} \geq \frac{2}{\overline{C}_{\nu}}
$$
and if ${C}_{\nu}$ exists, then
\begin{equation}\tag{\footnotesize{$\largestar \largestar \largestar$}}
  \underset{n \rightarrow \infty}{\lim}\frac{\log m_n^{\omega, \tilde{\omega}}(x,\tilde{x})}{-\log n} = \frac{2}{C_{\nu}} \ .
\end{equation}

\end{Th}

To prove this theorem, we will just apply Theorem \ref{firstresult} and Theorem \ref{secondresult} to the skew-product $S$ with a well-chosen observation, following the idea given in \cite{Rousseau}.

\begin{proof}
We use Theorem \ref{firstresult} and Theorem \ref{secondresult} for the dynamical system $(\Omega \times X, B(\Omega \times X), \mu, S)$ with the observation $f$ defined by
\begin{eqnarray*}
&& f: \Omega \times X \to X \\
&& \ \ \ \ \ \ (\omega,x) \mapsto x.
\end{eqnarray*}
{Indeed, with this particular observation $f$, studying the observed orbit of $(\omega,x)$ under the skew-product $S$ is similar to studying the random orbit of $x$ with respect to $\omega$ since
\[f(S^n(\omega,x))=f(\theta^n\omega,T^n_\omega(x))=T^n_\omega(x).\]}
Thus, for all $z \ \mbox{and} \ t \in  \Omega \times X$ we can link the shortest distance between two observed orbits and the shortest distance between two  random orbits. Set $z=(\omega,x)$ and $t=(\tilde{\omega},\tilde{x})$ then
\begin{eqnarray*}
m_n^f(z,t) &=& \min_{i,j=0, \ldots, n-1} \lt(d\lt(f\lt(S^i(\omega,x)\rt),f\lt(S^j(\tilde{\omega},\tilde{x})\rt)\rt)\rt) \\
&=& \min_{i,j=0, \ldots, n-1} \lt(d\lt(T_\omega^ix,T_{\tilde{\omega}}^j\tilde{x}\rt)\rt) \\
&=& m_n^{\omega, \tilde{\omega}}(x,\tilde{x}).
\end{eqnarray*}
Moreover, we can identify the pushforward measure:
$f_*\mu=\nu$. Therefore, in view of the lower and upper correlation dimensions, the following statement finishes the proof
$$\underline{C}_{f_*\mu}=\underline{C}_{\nu} \ \mbox{and} \ \overline{C}_{f_*\mu}=\overline{C}_{\nu}.$$

\end{proof}

In what follows, we present a collection of examples for which Theorem \ref{theoremrandom} holds. They illustrate some well-known random dynamical systems on the literature.  

\subsection{Non-i.i.d. random dynamical system} 

The first example is a non-i.i.d. random dynamical system for which it was computed recurrence rates in \cite{MaRousseau} and hitting times statistics in \cite{Rousseau}.

Consider the two linear maps which preserve Lebesgue measure $\Leb$ on $X = \mathbb{T}^1$, the one-dimensional torus:
\begin{eqnarray*}
&& T_1:X \to X      \quad \mbox{and} \quad  T_2: X \to X \\
&& \ \qquad x \mapsto 2x \hspace{2.1cm} x \mapsto 3x.
\end{eqnarray*}
The following skew product gives the dynamics of the random dynamical system:
\begin{eqnarray*}
&& S: \Omega \times X \to \Omega \times X \\
&& \ \ \ \ \ \ (\omega,x) \mapsto (\theta(\omega), T_\omega (x))
\end{eqnarray*}
with $\Omega=[0,1], \ T_\omega=T_1$ if $\omega \in [0,2/5)$ and $T_\omega=T_2$ if $\omega \in [2/5,1]$ where $\omega$ is the following piecewise linear map:

$$\theta(\omega)=\left\{\begin{array}{ll}
2\omega &  \text{ if } \, \omega\in [0,1/5) \\
3\omega-1/5 &  \text{ if } \, \omega \in [1/5,2/5) \\
2\omega-4/5 &  \text{ if }\, \omega \in [2/5,3/5) \\
3\omega/2-1/2 &  \text{ if } \,\omega \in [3/5,1].
\end{array}\right.$$

Note that the random orbit is constructed by choosing one of these two maps following a Markov process with the stochastic matrix

$$A=\left(\begin{array}{ll}
1/2 &  1/2 \\
1/3 &  2/3
\end{array}\right).$$

The associated skew-product $S$ is $\Leb \otimes \Leb$-invariant. It is easy to check that Lebesgue measure satisfies (a). Moreover, by \cite{Baladi} the skew product $S$ has an exponential decay of correlations. Since in this example $\nu=\Leb$, we have $C_\nu=1$ and Theorem \ref{theoremrandom} implies that for $\Leb \otimes \Leb \otimes \Leb \otimes \Leb$-almost every $(\omega,x, \tilde{\omega},\tilde{x}) \in [0,1]\times  \mathbb{T}^1 \times [0,1] \times  \mathbb{T}^1$,

$$  \underset{n \rightarrow \infty}{\lim}\frac{\log m_n^{\omega, \tilde{\omega}}(x,\tilde{x})}{-\log n} = 2.$$

\subsection{Randomly perturbed dynamical systems} 

Consider a deterministic dynamical system $(X, T, \mu)$ where $X$ is a compact Riemannian manifold, $T$ is a map and $\mu$ is a $T$-invariant probability measure.
We will present a random dynamical system constructed by perturbing the map $T$ with a random additive noise. For $\epsilon>0$, set $\Lambda_\epsilon=B(0,\epsilon)$ and let $\mathcal{P}_\epsilon$ be a probability measure on $\Lambda_\epsilon.$ For each $\omega \in \Lambda_\epsilon$, we denote the family of transformations $\{T_\omega\}_\omega$ where the map $T_\omega: X \to X$ are given by
$$T_\omega(x)=T(x)+\omega.$$
Denote $\mathcal{T}$ the i.i.d dynamical system on $X$ over $(\Lambda_\epsilon^{\mathds{N}}, \mathcal{P}_\epsilon^{\mathds{N}}, \sigma).$ In the case where $X=\mathbb{T}^d$, for some expanding and piecewise expanding maps, if $\epsilon$ is sufficiently small, it was proved (see e.g. \cite{AyFV,BaladiYoung,Viana}) that the random dynamical system has a stationary measure $\mu_\epsilon$ absolutely continuous with respect to Lebesgue measure with density $h_\epsilon$ such that $0< \underline{h}_\epsilon \leq h_\epsilon \leq \overline{h}_\epsilon < \infty$ and the system has an exponential decay of correlations.
Thus, since the assumptions (a) and (b) are satisfied one can apply Theorem \ref{theoremrandom} and obtain information on the behavior of the shortest distance $m_n^{\omega, \tilde{\omega}}$.

\subsection{Random hyperbolic toral automorphisms} 

A linear toral automorphism is a map $T:\mathbb{T}^2\rightarrow\mathbb{T}^2$ defined by the matrix action
$x\mapsto Ax$, where the matrix $A$ has integer entries and  $\det A=\pm 1$. We say that $T$ is hyperbolic if $A$ has
eigenvalues with modulus different from 1. For more simplicity, we will use the notation $A$ for both the matrix and the associated automorphism.

For an hyperbolic toral automorphism $A$, we denote $E_u^{A}$ the subspace spanned by $e_u^A$, the eigenvector associated to the eigenvalue whose absolute value is greater than 1 and we denote $E_s^{A}$ the subspace spanned by $e_s^A$, the eigenvector associated to the eigenvalue whose absolute value is less than 1.

Following the definition from \cite{AySt}, we say that a pair $(A_0,A_1)$ of hyperbolic toral automorphisms has the cone property if there exists an expansion cone $\mathcal{E}$ such that
\begin{enumerate}
 \item $A_i\mathcal{E}\subset \mathcal{E}$,
 \item there exists $\lambda_{\mathcal{E}}>1$ such that $\vert A_i x\vert\geq \lambda_{\mathcal{E}}\vert x\vert$ for $x\in\mathcal{E}$,
 \item  $E_u^{A_i}\cap \partial \mathcal{E}={0}$, where $\partial \mathcal{E}$ denote the boundary of $\mathcal{E}$,
\end{enumerate}
and there exists a contraction cone $\mathcal{C}$ such that $\mathcal{C}\cap \mathcal{E}={0}$ and
\begin{enumerate}
 \item $A_i^{-1}\mathcal{C}\subset \mathcal{C}$,
 \item there exists $\lambda_{\mathcal{C}}<1$ such that $\vert A_i^{-1} x\vert\geq \lambda_{\mathcal{C}}^{-1}\vert x\vert$ for $x\in\mathcal{C}$,
 \item $E_s^{A_i}\cap \partial \mathcal{C}={0}$.
\end{enumerate}

One can observe that for example a pair of hyperbolic toral automorphisms with positive entries, or a pair of
hyperbolic toral automorphisms with negative entries, has the cone property.

Let $\Lambda=\{0,1\}$ and $\theta=\sigma$ be the left shift on $\Lambda^{\mathds{N}}$. Let $A_0$, $A_1$ two hyperbolic automorphisms satisfying the cone property. Let $A_0$ be chosen with a probability $q$ and $A_1$ with a probability $1-q$, i.e. $\mathds{P}=\mathcal{P}^{\mathds{N}}$ with $\mathcal{P}(0)=q$ and $\mathcal{P}(1)=1-q$.

Then, for the i.i.d. random dynamical system on $\mathbb{T}^2$ over $(\Lambda^\mathds{N},\mathcal{P}^\mathds{N},\sigma)$, the Lebesgue measure is stationary (and thus hypothesis (a) is satisfied) and the system has an exponential decay of correlations (see \cite{AySt}).

Note that $\nu=\Leb \otimes \Leb$ implies that $C_\nu=2.$
Then, by Theorem \ref{theoremrandom} we get for $\mathds{P} \otimes\Leb\otimes \mathds{P} \otimes \Leb$-almost every $(\omega,x, \tilde{\omega},\tilde{x}) \in \Omega \times \mathbb{T}^2 \times \Omega \times \mathbb{T}^2,$

$$  \underset{n \rightarrow \infty}{\lim}\frac{\log m_n^{\omega, \tilde{\omega}}(x,\tilde{x})}{-\log n} = 1.$$

\subsection*{Acknowledgements}
This work is partially supported by CNPQ and FAPESB. AC is partially supported by CAPES. RL is partially supported by FAPESP (grant 2014/19805-1), and CNPQ PDJ grant (process 406324/2017-4). JR was partially supported by FCT project PTDC/MAT-PUR/28177/2017, with national funds, and by CMUP (UID/MAT/00144/2019), which is funded by FCT with national (MCTES) and European structural funds through the programs FEDER, under the partnership agreement PT2020. This work is part of the Universal FAPEMIG project "Din\^amica de recorr\^encia para shifts aleat\'orios, e processos com lapsos de mem\'oria" (process APQ-00987-18). {The authors would like to thank the anonymous referee for helpful comments, corrections and suggestions.}

\vspace{5cm}

  A.~Coutinho, \textsc{Instituto de Matem\'atica e Estat\'istica, Universidade de S\~ao Paulo, \\ Rua do Mat\~ao, 1010, CEP 05508-090, S\~ao Paulo-SP, Brazil}\par\nopagebreak
  \textit{E-mail address:}: \texttt{adcoutinho.s@gmail.com}

\vspace{1cm}

  R.~Lambert, \textsc{Faculdade de Matem\'atica, Universidade Federal de Uberl\^andia, \\ Av. Jo\~ao Naves de Avila, 2121, CEP 38408-100, Uberl\^andia-MG, Brazil}\par\nopagebreak
  \textit{E-mail address:} \texttt{rodrigolambert@ufu.br}

\textit{URL:} \texttt{https://sites.google.com/view/rodrigolambert}

\vspace{1cm}

  J.~Rousseau, \textsc{Departamento de Matem\'atica, Faculdade de Ci\^encias da Universidade do Porto, Rua do Campo Alegre, 687, 4169-007 Porto, Portugal}\par\nopagebreak
  \textsc{Departamento de Matem\'atica, Universidade Federal da Bahia, Av. Ademar de Barros s/n, 40170-110 Salvador, Brazil}\par\nopagebreak

  \textit{E-mail address:} \texttt{jerome.rousseau@ufba.br}
  
   \textit{URL:} \texttt{http://www.sd.mat.ufba.br/$\sim$jerome.rousseau}

\end{document}